\documentclass[review]{elsarticle}

\usepackage{lineno,hyperref}
\modulolinenumbers[5]

\journal{Discrete Applied Mathematics}

\usepackage{graphicx}
\usepackage{algorithm,algpseudocode}
\usepackage{mathrsfs}
\usepackage{amsmath}  
\usepackage{amsthm}
\usepackage{mathtools}
\usepackage{amssymb}
\usepackage{amsfonts} 
\usepackage{subcaption}
\usepackage[export]{adjustbox}
\usepackage{bbm}
\usepackage{cite}
\usepackage{url}
\usepackage{color}
\usepackage{comment}
\allowdisplaybreaks[1]

\newtheorem{theorem}{Theorem}
\newtheorem{proposition}{Proposition}
\newtheorem{lemma}{Lemma}
\newtheorem{corollary}{Corollary}

\newcommand{\mtc}{\mathcal}
\newcommand{\mb}{\mathbb}
\newcommand{\mbR}{\mathbb{R}}

\newcommand{\cS}{\mathcal{S}}

\newcommand{\F}{\mathcal{F}}

\newcommand{\CIJ}{\mtc{C}(I,J)}

\newcommand{\proj}{\textrm{proj}}

\newcommand{\cchiL}{\conv(\chi^L)}

\newcommand{\conv}{\textrm{conv}}

\newcommand{\beq}{\begin{equation}}
\newcommand{\eeq}{\end{equation}}
\newcommand{\ba}{\begin{aligned}}
\newcommand{\ea}{\end{aligned}}
\newcommand{\bdm}{\begin{displaymath}}
\newcommand{\edm}{\end{displaymath}}

\DeclarePairedDelimiterX\Set[2]{\lbrace}{\rbrace}{ #1 \,\delimsize| \,\mathopen{} #2 }

\begin{document}
\begin{frontmatter}
\title{Polyhedral Analysis of a Polytope from a Service Center Location Problem with 
a Type of Decision-Dependent Customer Demand}
\tnotetext[mytitlenote]{Polyhedral Analysis for a Service Center Location Problem}

\author[nwaddress]{Fengqiao Luo}
\cortext[mycorrespondingauthor]{Corresponding author}
\ead{fengqiaoluo2014@u.northwestern.edu}
\address[nwaddress]{Northwestern University, 
Department of Industrial Engineering and Management Science, Evanston, IL 60202}

\begin{abstract}
The polytope from a service center location problem with a special
location-dependent demand has been investigated in this paper. 
The location-dependent demand is defined based on a maximum-attraction-principle assumption. 
Along this direction, the intersection polytope of the capacity constraints 
and demand constraints yields a novel polyhedral structure. The structure is investigated
in this paper via a family of valid and facet-defining inequalities that can reflect the change
of demand fulfillment in a neighborhood.
\end{abstract}


\end{frontmatter}

\section{Introduction}
\label{sec:introd}
We investigate the polytope $\conv(\chi)$, where $\chi$
is the set of points satisfying the following linear and mixed-binary constraints:  
\beq\label{eqn:chi}
\chi=\Set*{[x,y]}{
\begin{array}{ll}
\displaystyle \sum^m_{i=1}x_{ij}\le c_jy_j & \quad\forall j\in[n], \\
\displaystyle \sum^n_{j=1}x_{ij} \le\max_{j\in[n]}d_{ij}y_j & \quad \forall i\in[m], \\
\displaystyle x_{ij}\ge0,y_j\in\{0,1\} & \quad\forall i\in[m],\forall j\in[n]
\end{array}
},
\eeq
where $[m]$ represents the set of indices $\{1,\ldots,m\}$. Note that
$\conv(\chi)$ is indeed a polytope, which is shown in Proposition~\ref{prop:conv(chi)-polytope}.

\subsection{Motivation and background}
The above linear system is motivated from a service center location
problem with standard capacity constraints and additional decision-dependent
constraints. More specifically, we consider a problem of allocating a set of
service centers (facilities) in a region to meet the customer demand from the region. 
Similarly to the traditional facility location problem, the region consists of 
customer sites and candidate locations for the service centers. 
Each customer site has certain demand that
needs to be fulfilled by a service center from its neighborhood.
Especially, we consider the case that the demand of each customer site 
is a value which depends on the location pattern of service centers. 
Suppose there are $m$ customer sites and $n$ candidate locations
for the service centers. Let $c_j$ be the capacity of the service center
at the location $j\in[n]$. Let $d_i(y)$ be the demand from customer site $i$
for a given service-center-location vector $y$. In particular, we consider 
a special type of dependency between the customer demand and the location
vector, which is simple but sensible in depicting customers' behavior in certain
situations. We referred it as the \emph{maximum attraction principle}.
To explain this principle, we assume that when there is just one service center
been allocated which is opened at $j\in[n]$, 
it will attract $d_{ij}$ amount of customer demand from the site $i$.
If there are multiple service centers opened at different locations, the maximum demand
that can be attracted from $i$ is equal to the maximum $d_{ij}$ for $j\in[n]$ that has a 
opened service center. This principle is formally described by the following equation:
\beq\label{eqn:max-attr}
d_i(y)=\left\{
\begin{array}{ll}
0 & \textrm{ if } y_j=0 \text{ for all }j\in[n], \\
d_{ij} & \textrm{ if } y=e_j \;\textrm{ for some } j\in[n], \\
\max_{j\in[n]} d_{ij}y_j & \textrm{ for other cases, } 
\end{array}
\right.
\eeq
where $e_j$ is the $n$-dimensional vector with the $j$-th entry being 1
and other entries being 0.

The maximum attraction principle matches 
with our intuition from practice in certain situations. Note that the principle
has incorporated the case that customers in site $i$ are not willing to
visit any service centers that are not within a certain range from $i$.
In this case, assume $F_i\subseteq[n]$ is the subset of candidate locations
that some customers from site $i$ are willing to visit, and then we can set
$d_{ij}=0$ for every $j\in[n]\setminus F_i$. In particular, $F_i$ can be a subset
of candidate locations in a neighborhood of site $i$, as customers usually 
will not consider visiting a service center that is beyond a certain distance 
from their living place. For a given location vector $y$, 
let $F_i(y)=\Set*{j\in F_i}{y_j=1}$.
The maximum attraction principle further assumes that within the preferable
opened service centers $F_i(y)$, there exists one service center $j^*$ (or multiple centers) that  
is (are) most attractive to the customers from site $i$ in the sense that 
$j^*\in\text{argmax}_{j\in F_i(y)}d_{ij}$, and the presence of multiple service
centers in $F_i(y)$ including $j^*$ attracts the same amount of demand from site $i$
as the presence of just a single service center at $j^*$. This assumption is also consistent
with our intuition in some situations. For example, If all service centers are identical 
in scale and service quality, the most attractive one to site $i$ is likely the one that is closest to $i$. 
Furthermore, customers who are willing to visit service centers with a longer distance are also 
willing to visit the one that is closest to their living place.

Notice that the polyhedral properties of the capacity constraints 
$\sum^m_{i=1}x_{ij}\le c_jy_j$ in \eqref{eqn:chi} 
have been well studied in \citep{1995-wolsey_cap-fac-loc-valid-ineqs}.
However, the demand attraction constraints $\sum^n_{j=1}x_{ij}\le\max_{j\in[n]}d_{ij}y_j$
concerned in this research are novel which brings new features to 
the polyhedral structure when interacting with
the capacity constraints. We focus on investigating the polyhedral structure 
induced by the demand attraction constraints and the capacity constraints.

\subsection{Literature review}
The facility location problem \citep{book-daskin-disc-loc} is one 
of the most fundamental problems investigated in operations research.
In this problem, a decision maker needs to decide locations 
of a limited number of facilities (factories, retail centers, power plants, service centers, etc.), 
and determine coverage of demand from different sites by the located facilities. 
The objective is to minimize the facility setup cost and the cost of production and delivery. 
This problem provides a basic framework to formulate related problems 
in resource allocation \citep{daskin2001-integ-fac-loc-transport-netwk-des}, 
supply chain management \citep{melo2009_fac-loc-supply-chain-managmt-revew}
and logistics \citep{boonmee2017_fac-loc-emerg-hum-logistics,
allen2012_relation-road-transport-fac-loc-log-managmt}, etc.
See \citep{daskin1998-strat-fac-loc-rev} for
a comprehensive review on this topic.

Facility location problems with non-deterministic demand have received 
plentiful investigation in the framework of stochastic optimization and 
robust optimization. In the stochastic optimization
framework, customer demand are independent random parameters with known probability
distributions and the problem can be formulated as
a two-stage stochastic program in which the location vector is in the 
first-stage decision that needs to be made before realization of 
customer demand \citep{1992-dual-stoch-fac-loc,2011-fac-loc-bern-dmd}. 
In a special case when the customer demand rates and service rates of each facility 
are assumed to follow exponential probability distributions,
the problem of minimizing long term average cost can be reformulated as   
a deterministic optimization problem with corresponding Poisson rates
to characterize the demand and service. This setting has been applied
to an allocation problem of ATMs \citep{2002-alg-fac-loc-stoch-demand}.  
The probability distribution of demand can also be used to define 
chance constraints to ensure a required service level under possible
stockout and supply disruption
\citep{2011-rob-fac-loc-bio-terror-attack,gulpinar2013-rob-fac-loc,
daskin2012_fac-loc-rand-disrupt-imp-estim,2017-multi-source-suppl-chance-constr}. 
In the robust optimization framework
for facility location problems,
the information of demand is partially known to the decision maker,
and the goal is to find a robust optimal location vector that optimizes
the objective after a worst-case realization of demand information
\citep{2006-fac-loc-uncert-rev}. 
Following this direction of research, 
Baron and Milner \citep{2010-fac-loc-rob-opt-approach} 
investigated a robust multi-period facility location problem
with box and ellipsoid sets of uncertainty.
Gourtani et. al. \citep{2020-gourtani_dro-two-stage-fac-loc} 
investigated a distributionally-robust two-stage facility location problem
with an ambiguity set defined corresponding to the mean and covariance matrix
of a random parameter for expressing the demand.
Facility location problems with demand uncertainty have been 
studied with a variety of novel application background, which includes
but not limited to medication coverage and delivery under a large-scale bio-terror attack
\citep{2011-rob-fac-loc-bio-terror-attack}, medical equipment (defibrillators) location 
problem to reduce cardiopulmonary resuscitation (CPR) risk \citep{2017-rob-deploy-card-arrest-loc-uncert},
humanitarian relief logistics \citep{2012-two-echelon-stoch-fac-loc}, 
and hazardous waste transportation \citep{2014-rob-fac-loc-hazard-waste-transport}, etc.

In a lot of real world problems, the parameters or the uncertainty of them
can interplay with the decision to be made. This behavior is well observed especially 
in a sequential (multi-stage) decision-making process,
in which information about system parameters are gradually revealed and the decisions made 
up until the current stage can reshape the uncertainty in future 
\citep{grossmann2006_stoch-prog-dec-uncert}. The customer demand
could also depend on the location of facilities or service centers
as revealed in \citep{eppli1996_location-shopping-center,rajagopal2009_behav-urban-shoppers},
which motivates the study of the facility location problem with location-dependent demand. 
As a robustification of it, Basciftci et. al. \citep{basciftci2020-dro-fl-endog-dmd} investigated
a distributionally-robust facility location problem with decision-dependent
ambiguity, where the ambiguity set is defined using disjointed lower and upper bounds on
the mean and variance of the candidate probability distributions for customer demand,
and these bounds are defined as linear functions of the location vector to admit
a tractable MILP reformulation.

Identifying valid and facet-defining inequalities is an essential topic
in understanding the polyhedral structure of a mixed-integer set of points 
\citep{conforti2014_integer-programming,wolsey1999_integer-comb-opt}.
As a by-product, adding some of these inequalities may improve the 
computational performance of solving the corresponding mixed-integer
linear program, as a result of strengthening the formulation 
\citep{pochet1993_lot-sizing-form-valid-ineq,
pochet1995_cap-fac-loc-valid-ineq,
perboli2010_new-valid-ineq-two-echelon-veh-routing,
coelho2014_inventory-routing-with-valid-ineq,
baldacci2009_valid-ineq-fleet-size-mix-veh-routing,
huygens2006_two-edge-hop-constr-netwk-design}.   
Some standard techniques are well developed to construct the 
face- an facet-defining inequalities such as lifting and simultaneous lifting
\citep{2007wolsey_lift-superadd-single-node-flow,
nemhauser2003_lifted-ineq-mip-thy-alg,
atamturk2003_facets-of-mixed-int-knapsack-poly,
gu1999_liftted-cover-ineq-complexity,gu1999_liftted-flow-cover-ineq,
gu1998_liftted-cover-ineq-computing}, 
lift-and-project \citep{au2016_analysis-poly-lift-proj-methd,
kocuk2016_cycle-base-form-valid-ineq-DC-power-transmission,
Luedtke2017_lift-and-proj-cuts-for-convex-minlp,burer2005_solve-lift-proj-relax,
aguilera2004_lift-proj-relax-matching-related,
balas1997_mod-lift-proj-proced}, and Fourier-Motzkin elimination
\citep{dantzig1972_FM-elimination-and-dual,williams1976_FM-elim-int-prog,
schechter1998_integration-polyhedron-appl-of-EM-elim}, etc.
By applying these techniques, different families of inequalities
such as lifted cover inequalities \citep{gu1998_liftted-cover-ineq-computing,
gu1999_liftted-cover-ineq-complexity}, 
lifted flow-cover inequalities \citep{gu1999_liftted-flow-cover-ineq},
and disjunctive inequalities \citep{balas1993_lift-project-cut-plane}, etc. are constructed 
for corresponding basic polytopes that are often building blocks in the formulation
of various mixed-integer linear programming problems in practice. 
In many mixed-integer linear programming problems from the application, 
construction of valid and facet-defining inequalities usually depends on 
ad hoc observations that identify sources leading to fractional solutions [cite references]. 
The key observation for constructing facet-defining inequalities in this paper
is inspired by \citep{wolsey1995_cap-fac-loc-ineq}. 
The work in \citep{wolsey1995_cap-fac-loc-ineq} has constructed and generalized
a family of facet-defining inequalities for the polytope of a standard capacitated facility
location problem. Computational performance with the inequalities added to 
the linear relaxation problems have been investigated in 
\citep{aardal1998_cap-fac-loc-sep-alg-comp,labbe2004_branch-cut-alg-hub-loc}.
In the standard capacitated facility location problem considered in 
\citep{wolsey1995_cap-fac-loc-ineq}, the demand of each client site is 
given as a fixed constant, whereas the customer demand is considered 
to be decision-dependent in this work.

\section{Polyhedral properties of $\conv(\chi)$}
\label{sec:poly-property}
We first show in the following proposition that the convex hull 
$\conv(\chi)$ is indeed polytope.
\begin{proposition}\label{prop:conv(chi)-polytope}
The convex hull $\conv(\chi)$ is a polytope.
\end{proposition}
\begin{proof}
Let $\chi(\hat{y})$ be the subset of $\chi$ with a given binary
vector $y$ at value $\hat{y}$. Once the variable $y$ is fixed,
the right side of the inequalities $\sum^m_{i=1}x_{ij}\le c_j\hat{y}_j$
and $\sum^n_{j=1}x_{ij}\le\max_{j\in[n]} d_{ij}\hat{y}_j$ becomes
a constant, and hence $\chi(\hat{y})$ is a polytope. Notice that
the convex hull of $\chi$ can be rewritten as 
\beq\label{eqn:conv(chi)-rewritten}
\conv(\chi)=\conv\left(\cup_{\hat{y}\in\{0,1\}^n}\chi(\hat{y})\right).
\eeq
The right side of \eqref{eqn:conv(chi)-rewritten} is a convex hull of
a finite union of polytopes, and hence the result is also a polytope.
\end{proof}
Notice that the representation \eqref{eqn:chi} of $\chi$ involves a nonlinear
term $\max_{j\in[m]}d_{ij}y_j$ of the decision vector $y$, which is inconvenient
for the polyhedral analysis. The nonlinear term can be avoided by reformulating
$\chi$ in a lifted space. The lifted reformulation of $\chi$ is given as follows:
\beq\label{eqn:chi_L}
\chi^L=\Set*{[x,y,q]}{
\begin{array}{ll}
\displaystyle \sum^m_{i=1}x_{ij}\le c_jy_j & \quad\forall j\in[n], \\
\displaystyle \sum^n_{j=1}x_{ij} \le\sum_{j\in[n]}d_{ij}q_{ij} & \quad \forall i\in[m], \\
\displaystyle \sum^n_{j=1}q_{ij}\le 1 & \quad \forall i\in[n], \\
\displaystyle q_{ij}\le y_j\; & \quad  \forall i\in[n],\; \forall j\in[m], \\ 
\displaystyle x_{ij}\ge0,\;q_{ij}\ge 0,\; y_j\in\{0,1\} & \quad\forall i\in[m],\forall j\in[n]
\end{array}
}.
\eeq
\begin{proposition}\label{prop:chi_chiL}
The polytope $\chi$ is the projection of $\chi^L$ on to the $[x,y]$ space,
i.e., $\emph{conv}(\chi)=\emph{proj}_{[x,y]}\emph{conv}(\chi^L)$.
\end{proposition}
\begin{proof}
Notice that both $\chi$ and $\chi^L$ can be represented as
the convex hull of a finite number of sub-polytopes as follows:
\beq
\ba
&\conv(\chi)=\conv\left(\cup_{\hat{y}\in\{0,1\}^n}\chi(\hat{y})\right), \\
&\conv(\chi^L)=\conv\left(\cup_{\hat{y}\in\{0,1\}^n}\chi^L(\hat{y})\right).
\ea
\eeq
So it suffices to show that $\chi(\hat{y})=\proj_x \chi^L(\hat{y})$.
First, consider any $[\hat{x},\hat{y},\hat{q}]\in\chi^L(\hat{y})$.
To show $[\hat{x},\hat{y}]\in\chi(\hat{y})$, it suffices to verify that
$\sum^n_{j=1}\hat{x}_{ij} \le\max_{j\in[n]}d_{ij}\hat{y}_j$ holds for
every $i\in[n]$. Indeed, if we let $J_i=\Set*{j\in[m]}{\hat{y}_j=1}$,
we have
\beq
\ba
&\sum^n_{j=1}\hat{x}_{ij}\le\sum_{j\in[n]}d_{ij}\hat{q}_{ij}
=\sum_{j\in J_i}d_{ij}\hat{q}_{ij}+\sum_{j\in[m]\setminus J_i}d_{ij}\hat{q}_{ij} \\
&\le\sum_{j\in J_i}d_{ij}\hat{q}_{ij}+\sum_{j\in[m]\setminus J_i}d_{ij}\hat{y}_j 
= \sum_{j\in J_i}d_{ij}\hat{q}_{ij}\\
&\le\Big(\max_{j\in J_i}d_{ij}\Big) \sum_{k\in J_i}\hat{q}_{ik} 
\le \max_{j\in J_i}d_{ij}=\max_{j\in[n]}d_{ij}\hat{y}_j.
\ea
\eeq
Therefore, we have shown that $\proj_x \chi^L(\hat{y})\subseteq\chi(\hat{y})$.
Second, consider any $[\hat{x},\hat{y}]\in\chi(\hat{y})$. We need to show that
there exists a $\hat{q}$ such that $[\hat{x},\hat{y},\hat{q}]\in\chi^L(\hat{y})$.
We can construct a $\hat{q}$ and the as follows: 
Let $k_i\in\text{argmax}_{j\in J_i}d_{ij}$. For any $i\in[m]$, let $\hat{q}_{ik_i}=1$
and $\hat{q}_{ij}=0$ for all $j\in[m]\setminus\{k_i\}$. Clearly the constraints
$\sum^n_{j=1}\hat{q}_{ij}\le 1$ and $\hat{q}_{ij}\le\hat{y}_j$ are satisfied for 
each $i\in[m]$ by the way of constructing $\hat{q}$. Furthermore, we have
\beq
\sum_{j\in[n]}d_{ij}\hat{q}_{ij}=d_{i,k_i}=\max_{j\in J_i}d_{ij}=\max_{j\in[n]}d_{ij}\hat{y}_{j}
\ge\sum^n_{j=1}\hat{x}_{ij},
\eeq
which shows that the constraint $\sum^n_{j=1}\hat{x}_{ij}\le\sum_{j\in[n]}d_{ij}\hat{q}_{ij}$
is satisfied, and hence $\chi(\hat{y})\subseteq\proj_x\chi^L(\hat{y})$.
\end{proof}
Proposition~\ref{prop:chi_chiL} has revealed a type of equivalence between
$\chi$ and $\chi^L$. We will hence focus on deriving valid and facet-defining 
inequalities for $\chi^L$. As we will see later, these inequalities do not involve the 
variables $q$, and hence they are also valid for $\chi$.

To give an insight on how the valid and facet-defining inequalities are constructed, 
we consider the quantity $\sum_{i\in I}\sum_{j\in J}x_{ij}$ for any subsets 
$I\subseteq[m]$ and $J\subseteq[n]$. 
The capacity and demand attraction constraints
imply that $\sum_{i\in I}\sum_{j\in J}x_{ij}\le\sum_{j\in J}c_jy_j$
and $\sum_{i\in I}\sum_{j\in J}x_{ij}\le\sum_{i\in I}\max_{j\in[n]}d_{ij}y_j$.
The upper bound of the quantity 
$\sum_{i\in I}\sum_{j\in J}x_{ij}$ depends
on which constraint is more restrictive. Figure~\ref{fig:dmd-cap-compete}
illustrates a critical situation that if any one of the service centers in $J$
is not opened, the total attracted demand is greater than the total capacity.
While if all service centers in $J$ are opened, the total attracted demand is
less than the total capacity. Based on this critical situation, we can derive a
facet-defining valid inequality $\conv(\chi^L)$ under some mild conditions. 
The theoretical results on the polyhedral properties of $\conv(\chi^L)$ are 
given by Theorem~\ref{thm:valid-ineq-Z} and Theorem~\ref{thm:facet-critical-IJ}.

\begin{figure}
\centering
\includegraphics[trim=1cm 17.6cm 4cm 4cm,width=0.9\textwidth]{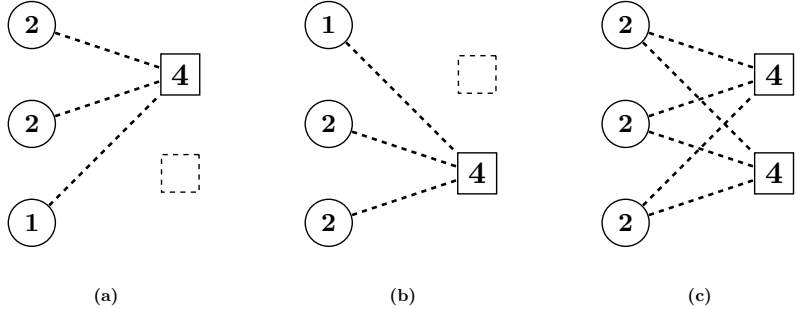}
\caption{\footnotesize An illustration of the competition between the capacity constraints and the demand attraction constraints.
In all subfigures, three customer sites are represented as circles which are denoted as S1, S2, S3 (from top to bottom) 
and two candidate service center locations are represented as squares which are denoted as F1, F2 (from top to bottom).
The number in a customer site (circle) indicates the attracted demand, and the number in a service center (square)
indicates the capacity. In case (a) only F1 is opened which attracts 2 units of demand from S1 and S2 respectively
and 1 unit from S3. In this case the total attracted demand is greater than the total capacity. In case (b) only F2 is opened
and the situation is similar to (a). In case (c) two service centers are opened, which attracts 2 units of demand from 
all customer sites. The total attracted demand is less than the total capacity.
}\label{fig:dmd-cap-compete}
\end{figure}

\subsection{Technical lemmas for proving facetness}
\label{sec:lemmas}
We provide a sequence of technical lemmas that will be used to prove the 
facet-extension results given in Theorem~\ref{thm:facet-extension}.
\begin{lemma}\label{lem:n+1->n}
Let $n$ be a positive integer. Then there exist $n$ linearly independent vectors in
any $n+1$ affinely independent vectors.
\end{lemma}
\begin{proof}
Suppose $S=\{v^k\}^n_{k=0}$ are $n+1$ affinely independent vectors.
If $\textrm{span}(S)\ge n$, then there must exist $n$ linearly independent vectors in $S$.
We now suppose that $\textrm{span}(S)\le n-1$. Specifically, suppose $\textrm{span}(S)=m$,
with $m\le n-1$. Then there exist $m$ linearly independent vectors in $S$. We assume that
$\{v^k\}^n_{k=n-m+1}$ are linearly independent without loss of generality.
It follows that $v^i\in\textrm{span}\{v^k\}^n_{k=n-m+1},\;\forall i\in\{0,1,\ldots,n-m\}$.  
Therefore, we have $\textrm{span}\{v^k-v^0\}^n_{k=1} \subseteq \textrm{span}\{v^k\}^n_{k=n-m+1}$,
and hence $\dim \{v^k-v^0\}^n_{k=1}\le m$,
which contradicts to that $\{v^k-v^0\}^n_{k=1}$ are linearly independent.
\end{proof}

\begin{lemma}\label{lem:aff-indpt}
Let $\{v_i\}^k_{i=1}$ be affinely independent vectors in $\mathbb{R}^n$ and $\{u_i\}^m_{i=1}$ be linearly
independent vectors in $\mathbb{R}^n$. If $\emph{span}\{u_i\}^m_{i=1} \cap \emph{span} \{v_i\}^k_{i=1}=\{0\}$,
then $\{v_i\}^k_{i=1}\cup \{u_i\}^m_{i=1}$ are affinely independent vectors in $\mathbb{R}^n$. \newline
\end{lemma}
\begin{proof}
Consider any coefficients $\{a_i\}^k_{i=2}$ and $\{b_i\}^m_{i=1}$ that satisfy the following equation
\begin{equation}\label{eqn:sum_av_bu}
\sum^k_{i=2}a_i(v_i-v_1)+\sum^m_{j=1}b_j(u_j-v_1)=0.
\end{equation}
If $b_j=0\;\;\forall j\in[m]$, then \eqref{eqn:sum_av_bu} implies $\sum^k_{i=2}a_i(v_i-v_1)=0$,
and hence $a_i=0\;\;\forall i\in \{2,\ldots,k\}$ due to the affinely independence of $\{v_i\}^k_{i=1}$.
Now let $J=\{j\in[m]:\;b_j\neq 0\}$ and assume $|J|\ge 1$, then \eqref{eqn:sum_av_bu} becomes
\begin{equation}
\sum_{j\in J}b_ju_j = \sum_{j\in J}b_jv_1-\sum^k_{i=2}a_i(v_i-v_1).
\end{equation}
By linearly independence, we have $\sum_{j\in J}b_j u_j\neq 0$. Since \newline 
$\textrm{span}\{u_i\}^m_{i=1} \cap \textrm{span}\{v_i\}^k_{i=1}=\{0\}$,
we must have $\sum_{j\in J}b_ju_j= 0$, which leads to a contradiction.
Therefore, we have proved that  $a_i=0\;\;\forall i\in \{2,\ldots,k\}$ and $b_j=0\;\;\forall j\in[m]$,
which shows that $\{v_i\}^k_{i=1}\cup \{u_i\}^m_{i=1}$ are affinely independent vectors. 
\end{proof}

\begin{lemma}\label{lem:u+ev_ind}
Let $l,m,n$ be three non-negative integers such that $m\ge \emph{max}\{n,l\}$ and $n>0$. 
Let $\{v_k\}^n_{k=1}$, 
$\{v^{\prime}_k\}^l_{k=1}$ and $\{u_k\}^n_{k=1}$ be three sets of vectors in $\mathbb{R}^m$.
The following properties hold.
\newline
\emph{(a)} If $\{v_k\}^n_{k=1}$ and $\{v^{\prime}_k\}^l_{k=1}$ are $n+l$ linearly independent vectors, 
then for any $r>0$ there exists an $0<\epsilon<r$ such that  
$\{u_k+\epsilon v_k\}^n_{k=1}$ and $\{v^{\prime}_k\}^l_{k=1}$ 
are also $n+l$ linearly independent vectors. 
\newline
\emph{(b)} If $\{v_k\}^n_{k=1}$ are $n$ affinely independent vectors, 
$\{v^{\prime}_k\}^l_{k=1}$ are $l$ linearly independent vectors,
$\emph{span}\{v^{\prime}_k:\forall k\in[l]\}\subseteq\emph{span}\{u_k:\forall k\in[n]\}$
and $\emph{span}\{v_k:\forall k\in[n]\}\cap\emph{span}\{u_k:\forall k\in[n]\}=\{0\}$,
then for any $r>0$ there exists an $0<\epsilon<r$ such that   
$\{u_k+\epsilon v_k\}^n_{k=1}$ and $\{v^{\prime}_k\}^l_{k=1}$ 
are $n+l$ affinely independent vectors.
\end{lemma}
\begin{proof}
(a) Note that the set $\{v_k\}^n_{k=1}\cup\{v^{\prime}_k\}^l_{k=1}$ contains $n+l$ linearly independent vectors
if and only if $\det A^{\top}A\neq 0$, where $A=[v_1,\ldots,v_n,v^{\prime}_1,\ldots,v^{\prime}_l]$. 
It amounts to show that there exists an $0<\epsilon<r$ such that $\det B^{\top}(\epsilon)B(\epsilon)\neq 0$, 
where $B(\epsilon)=[u_1+\epsilon v_1,\ldots,u_n+\epsilon v_n,v^{\prime}_1,\ldots,v^{\prime}_l]$.
We prove it by contradiction. Assume that $\det B^{\top}(\epsilon)B(\epsilon)= 0$ for all $0<\epsilon<r$.
Note that since $\det B^{\top}(\epsilon)B(\epsilon)$ is a polynomial of $\epsilon$ on $\mathbb{R}$,
this hypothesis implies that the polynomial $\det B^{\top}(\epsilon)B(\epsilon)$ has infinite number of real roots.
However, only zero polynomial has infinitely many real roots. If we can show that $\det B^{\top}(\epsilon)B(\epsilon)$ 
is not a zero polynomial, we deduce a contradiction. The polynomial  $\det B^{\top}(x)B(x)$
can be rewritten as
\begin{equation}
\begin{aligned}
\det B^{\top}(x)B(x) &= \det \big([u^{\top}_1+xv^{\top}_1,\ldots,u^{\top}_n+xv^{\top}_n,v^{\prime T}_1,\ldots,v^{\prime T}_l]^{\top}\\
&\qquad\quad [u_1+xv_1,\ldots,u_n+xv_n,v^{\prime }_1,\ldots,v^{\prime}_l]\big) \\
&=(\det A^{\top}A)x^{2n} + f(x),
\end{aligned}
\end{equation}  
where $f(x)$ is a polynomial of $x$ with degree at most $2n-1$. Therefore we have
\begin{displaymath}
\lim_{x\to\infty}\frac{\det B^{\top}(x)B(x)}{x^{2n}} = \det A^{\top}A\neq 0,
\end{displaymath} 
which shows that $\det B^{\top}(x)B(x)$ is not a zero polynomial.

(b) We first make the following claim.
\emph{Claim}: For any $r>0$ there exists an $0<\epsilon<r$ 
such that the set $\{u_k+\epsilon v_k\}^n_{k=1}$
contains $n$ affinely independent vectors. 
\emph{Proof of Claim}: Note that the vectors $\{v_k\}^n_{k=1}$ 
are affinely independent if and only if $\{[1,v^{\top}_k]\}^n_{k=1}$
are linearly independent. 
Since by assumption we have $\textrm{span}\{v_k:\forall k\in[n]\}\cap\textrm{span}\{u_k:\forall k\in[n]\}=\{0\}$,
it amounts to show that there exists an $0<\epsilon<r$ such that 
$\{[1,u^{\top}_k+\epsilon v^{\top}_k]\}^n_{k=1}$ 
are linearly independent, which is equivalent to
that $\{[\epsilon,u^{\top}_k+\epsilon v^{\top}_k]\}^n_{k=1}$ are linearly independent.
Since we have $[\epsilon,u^{\top}_k+\epsilon v^{\top}_k]=[0,u^{\top}_k]+\epsilon[1,v^{\top}_k]$,
we can apply the result of Part (a) to conclude the proof the claim.

Suppose $\{u_k+\epsilon v_k\}^n_{k=1}$ are $n$ affinely independent vectors. We now show
that $\textrm{span}\{u_k+\epsilon v_k:\forall k\in[n]\}\cap\textrm{span}\{v^{\prime}_k:\forall k\in[l]\}=\{0\}$.
Then by Lemma~\ref{lem:aff-indpt}, we can deduce that 
$\{u_k+\epsilon v_k\}^n_{k=1}$ and $\{v^{\prime}_k\}^l_{k=1}$ 
are $n+l$ affinely independent vectors.
Assume that $\textrm{span}\{u_k+\epsilon v_k:\;\forall k\in[n]\}\cap\textrm{span}\{v^{\prime}_k:\;\forall k\in[l]\}\neq\{0\}$.
Then there exist coefficients $\{\alpha_k\}^n_{k=1}$ and $\{\beta_k\}^l_{k=1}$ which are not all zero, such that
$\sum^n_{k=1}\alpha_k(u_k+\epsilon v_k)=\sum^l_{k=1}\beta_kv^{\prime}_k$,
which implies that $\epsilon\sum^n_{k=1}\alpha_kv_k=\sum^l_{k=1}\beta_kv^{\prime}_k-\sum^n_{k=1}\alpha_ku_k$.
Since $\sum^n_{k=1}\alpha_kv_k\in\textrm{span}\{v_k:\forall k\in[n]\}$ and 
$\sum^l_{k=1}\beta_kv^{\prime}_k-\sum^n_{k=1}\alpha_ku_k\in\textrm{span}\{u_k:\forall k\in[n]\}$.
But $\textrm{span}\{v_k:\forall k\in[n]\}\cap\textrm{span}\{u_k:\forall k\in[n]\}=\{0\}$ holds by the given hypothesis.
It follows that $\sum^n_{k=1}\alpha_kv_k=0$ and hence $\alpha_k=0\;\;\forall k\in[n]$ and $\beta_k=0\;\;\forall k\in[l]$,
which leads to a contradiction. 
\end{proof}

\begin{lemma}\label{lem:S^l_W_aff_ind}
\emph{(a)} Let $n$ be a positive integer. The following set of vectors are affinely independent:
$U=\{e^n_i-e^n_{i+1}: \forall i\in[n]\}$ with the convention $n+1\to 1$,
where $e^n_i$ is a $n$-dimensional vector with the $i$-th entry being 1 and 
other entries being 0.
\newline
\emph{(b)} Let $\{n_i\}^k_{i=1}$ be $k$ positive integers with $n_i\ge 2\;\;\forall i\in[k]$. 
Consider the following sets of vectors in $\mathbb{R}^N$ with $N=\sum^k_{i=1}n_i$:
\begin{displaymath}
\begin{aligned}
&U^l=\{[0^{n_1},\ldots,0^{n_{l-1}},e^{n_l}_i-e^{n_l}_{i+1},0^{n_{l+1}},\ldots,0^{n_k}]:\;\forall i\in[n_l-1]\} 
\quad \forall l\in[k],\\
&V = \{[0^{n_1},\ldots,0^{n_{i-1}},e^{n_i}_1,-e^{n_{i+1}}_1,0^{n_{i+2}}\ldots,0^{n_k}]:\;\forall i\in[k]  \},
\end{aligned}
\end{displaymath}
with the convention $k+1\to 1$. The $N$ vectors in $\left(\cup^k_{l=1}U^l\right)\cup V$ are affinely independent.
\end{lemma}
\begin{proof}
Part (a) can be easily verified. We now focus on proving Part (b).
Without loss of generality, we prove for the case that $n_l\ge4$ for all $l\in[k]$.
It is easy to verify that all vectors in $\cup^k_{l=1}U^l$ are linearly independent, and by Part (a),
the vectors in $V$ are affinely independent. We claim that 
$\textrm{span}(V)\cap \textrm{span}\{\cup^k_{l=1}U^l\}=\{0^N\}$.  
If the claim holds then by Lemma~\ref{lem:aff-indpt}, 
the $N$ vectors in $\left(\cup^k_{l=1}U^l\right)\cup W$ are affinely independent. 
To prove this claim, we label the vectors in $U^l$ as 
$u_{li}$ for $l\in[k]$ $i\in[n_l-1]$, and label the vectors in $V$
as $v_l$ for $l\in[k]$. We can verify that $\{v_l\}^{k-1}_{l=1}$
are linearly independent and $v_k=-\sum^{k-1}_{l=1}v_l$.
Suppose $\textrm{span}(V)\cap \textrm{span}\{\cup^k_{l=1}U^l\}\neq\{0^N\}$
and let $\Set*{\alpha_{li}}{l\in[k],\;i\in[n_l-1]}$ and $\Set*{\beta_l}{l\in[k-1]}$
be the coefficients such that 
\beq
w=\sum_{l\in[k]}\sum_{i\in[n_l-1]}\alpha_{li}u_{li}=\sum_{l\in[k-1]}\beta_lv_l\neq 0^N.
\eeq
Consider the $e^{n_s}_i$ entry of $w$ for $s\in[k]$ and $i\in\{2,\ldots,n_l-1\}$.
Evaluating this coefficient based on $\sum_{l\in[k-1]}\beta_lv_l$ implies that
the coefficient should be 0.
On the other side, evaluating this coefficient based on 
$\sum_{l\in[k]}\sum_{i\in[n_l-1]}\alpha_{li}u_{li}$ implies that it
should be equal to $\alpha_{s,i}-\alpha_{s,i-1}$. Therefore,
we must have $\alpha_{s,i}=\alpha_{s,i-1}$ for $s\in[k]$ and $i\in\{2,\ldots,n_l-1\}$.
For $l\in[k]$, let $a_l=\alpha_{l,i}$ for $i\in\{2,\ldots,n_l-1\}$.
By matching the coefficient corresponding to the $e^{n_1}_1$ entry of $w$,
we get $a_1=\beta_1$.  By matching the coefficient corresponding to 
the $e^{n_l}_{n_l}$ entry of $w$, we get $a_l=0$ for $l\in[k]$.
This shows that $w=0$, which leads to a contradiction.
\end{proof}

\begin{lemma}\label{lem:lin-indpt}
Let $l,m,n$ be three positive integers. 
Let $u^k\in\mb{R}^l\otimes\mb{R}^n$ and $v^k\in\mb{R}^m\otimes\mb{R}^n$ for all $k\in[n]$
be arbitrary vectors with dimension $l\times n$ and $m\times n$, respectively.
Let $u^{\prime}\in\mb{R}^{l}\otimes\mb{R}^n$, 
and $v^{\prime}\in\mb{R}^{m}\otimes\mb{R}^n$ be two arbitrary vectors.
Let $\{w^i\}^n_{i=1}$ be $n$ linearly independent vectors in $\mathbb{R}^n$ with $w^i_j\in\{0,1\}\;\;\forall i,j\in[n]$. Let 
$J_i=\{j\in[n]: w^i_j=1\}$. Consider the following sets of vectors in the 
$\mb{R}^{l\times n}\oplus\mb{R}^{m\times n}\oplus\mb{R}^n$ space:
\begin{displaymath}
\begin{aligned}
&U^k=\{[u^k,v^k,w^k]+[\epsilon e_{ij},0,0]-[\epsilon e_{i^{\prime}j},0,0]:
	\forall i,i^{\prime}\in [l],\;i\neq i^{\prime},\;\forall j\in J_k \} \\
	&\qquad\qquad \forall k\in[n], \\
&V^k=\{[u^k,v^k,w^k]+[0,\epsilon e_{ij},0]:
	\forall i\in [m],\;\forall j\in J_k \} \quad \forall k\in[n], \\
&Y=\{[u^{\prime}, v^{\prime},1^n]+[\epsilon e_{1j},0,0]-[\epsilon e_{1,j+1},0,0]:
       \forall j\in[n] \}\\ 
       &\qquad\qquad \textrm{with the convention } n+1\to 1, \\
&W=\{(u^k, v^k, w^k): \forall k\in[n]\},
\end{aligned}
\end{displaymath}
where $1^n$ is the $n$-dimensional vector with all entries being 1,
$[e_{ij},0,0]:=[e^l_i\otimes e^n_j,0^{m\times n},0^n]$ and 
$[0,e_{ij},0]:=[0^{l\times n},e^m_i\otimes e^n_j,0^n]$.
Then the set $S=\cup^n_{k=1}(U^k\cup V^k)\cup Y\cup W$ contains $(l+m+1)n$ affinely independent vectors.
\end{lemma}
\begin{proof}
Since $\{w^i\}^n_{i=1}$ are $n$ linearly independent binary vectors in $\mathbb{R}^n$,
for each $j\in[n]$, there exists an $i\in[n]$ (depending on $j$) such that $j\in J_i$.
Therefore, there exists a mapping $\pi:[n]\to[n]$ such that $\pi$ is surjective and $j\in J_{\pi(j)}$.
We define the following sets of vectors:
\begin{displaymath}
\begin{aligned}
&U^{\prime j}=\{[u^{\pi(j)},v^{\pi(j)},w^{\pi(j)}]+[\epsilon e_{ij},0,0]-[\epsilon e_{i+1,j},0,0]:
	\forall i\in [l-1]\} &\quad \forall j\in[n] \\
&V^{\prime j}=\{[u^{\pi(j)},v^{\pi(j)},w^{\pi(j)}]+[0,\epsilon e_{ij},0]:
	\forall i\in [m] \} &\quad \forall j\in[n]
\end{aligned}
\end{displaymath}
Clearly, the sets $\{U^{\prime j}:\forall j\in[n]\}$ are disjoint and $\cup_{j\in[n]}U^{\prime j}\subseteq \cup_{k\in[n]}U^k$.  
Similarly, the sets $\{V^{\prime j}:\forall j\in[n]\}$ are disjoint and $\cup_{j\in[n]}V^{\prime j}\subseteq \cup_{k\in[n]}V^k$.
We claim that all vectors in the following sets are affinely independent:
\begin{displaymath}
\begin{aligned}
&U^{\prime\prime}=\{[e_{ij},0,0]-[e_{i+1,j},0,0]: \forall i\in [l-1],\;\forall j\in[n]\},  \\
&V^{\prime\prime}=\{[0,e_{ij},0]: \forall i\in [m],\forall j\in[n] \}, \\
&Y^{\prime\prime}=\{[e_{1j},0,0]-[e_{1,j+1},0,0]:\forall j\in[n] \}.
\end{aligned}
\end{displaymath}
The claim can be proved using Lemma~\ref{lem:S^l_W_aff_ind}(b) and Lemma~\ref{lem:aff-indpt} with following argument.
By Lemma~\ref{lem:S^l_W_aff_ind}(b), all vectors in $U^{\prime\prime}\cup Y^{\prime\prime}$ are affinely independent.
One can easily verify that all vectors in $V^{\prime\prime}$ are independent and $\textrm{span}\{U^{\prime\prime}\cup Y^{\prime\prime}\}
\cap\textrm{span}\{V^{\prime\prime}\}=\{[0,0,0]\}$, and hence by Lemma~\ref{lem:aff-indpt} all vectors in 
$U^{\prime\prime}\cup V^{\prime\prime}\cup Y^{\prime\prime}$ are affinely independent. Then by Lemma~\ref{lem:u+ev_ind}(b),
the following $(l+m+1)n$ vectors are affinely independent:
\begin{displaymath}
\begin{aligned}
&U^*=\{[u^{\pi(j)},v^{\pi(j)},w^{\pi(j)}]+[\epsilon e_{ij},0,0]-[\epsilon e_{i+1,j},0,0]:\forall i\in [l-1],\;\forall j\in[n]\},  \\
&V^*=\{[u^{\pi(j)},v^{\pi(j)},w^{\pi(j)}]+[0,\epsilon e_{ij},0]: \forall i\in [m],\;\forall j\in[n] \}, \\
&Y^*=\{[u^{\prime},v^{\prime},1^n]+[\epsilon e_{1j},0,0]-[\epsilon e_{1,j+1},0,0]:\forall j\in[n] \},\\
&W=\{[u^k,v^k,w^k]:\forall k\in[n]\}.
\end{aligned}
\end{displaymath} 
Since we have $U^*\cup V^*\cup Y^*\cup W\subseteq S$, 
it implies that $S$ should contain $(l+m+1)n$ affinely independent vectors.
\end{proof}

\subsection{Valid and facet-defining inequalities}
Our goal is to construct valid and facet-defining inequalities for the convex hull $\CIJ$  
induced by the subsets $I, J$ of indices for the customer sites
and candidate locations, and then show that some of them are also valid and facet-defining 
inequalities for $\conv(\chi^L)$ and $\conv(\chi)$: 
\beq\label{def:ZIJ}
\CIJ=\Set*{[x,y,q]\in\conv(\chi^L)}{y_j=0\;\forall j\in[n]\setminus J,\;x_{ij}=0,\;q_{ij}=0\;\forall(i,j)\notin I\times J}.
\eeq
The conversion from the facets of $\CIJ$ to the facets of $\conv(\chi^L)$ is depicted by the following theorem:
\begin{theorem}[Facet extension]\label{thm:facet-extension}
Consider the polytope $\cchiL$ defined in \eqref{eqn:chi_L} induced by
the index subsets $I\subseteq[m]$ and $J\subseteq[n]$. 
Suppose $\cchiL$ is in full dimension, i.e., $\dim\cchiL=(2m+1)n$,
and $d_{ij}>0$ for $i\in[m],\; j\in[n]$. Suppose an inequality 
\beq\label{eqn:gen-facet}
\sum_{j\in J}\alpha_jy_j+\sum_{i\in I}\sum_{j\in J}(x_{ij}+\beta_{ij}q_{ij})\le \gamma
\eeq 
with coefficients $\alpha_j$, $\beta_{ij}$ and $\gamma$ define a facet of $\CIJ$ and it is
valid for $\cchiL$. If there exists a set $V$ of $(2|I|+1)|J|$ affinely independent points in 
$\CIJ$ that are on the plane defined by this inequality and $V$ satisfies the following conditions:
\begin{itemize}
	\item[\emph{(1)}] For any point $[x,y,q]\in V$, it holds that $y_j\in\{0,1\}\;\forall j\in J$; 
	\item[\emph{(2)}] $\forall j\in J$ there exists an point $[x^{(j)},y^{(j)},q^{(j)}]\in V$ such that
				  $y^{(j)}_j=1$ and $\sum_{i\in I}x^{(j)}_{ij}<c_j$;
	\item[\emph{(3)}] For each $i\in I$, there exists a point 
				 $[\tilde{x}^{(i)},\tilde{y}^{(i)},\tilde{q}^{(i)}]\in V$ such that \newline
				 $\sum_{j\in J}q^{(i)}_{ij}<1$.
\end{itemize}
Then the inequality \eqref{eqn:gen-facet} defines a facet of $\cchiL$.
\end{theorem}
\begin{proof}
We first show that $\cchiL$ is of full dimension. 
Since $\cchiL$ is defined by $(2m+1)n$ variables, 
we will prove $\dim\cchiL=(2m+1)n$ by constructing $(2m+1)n+1$ 
affinely independent points in $\cchiL$ as follows:
\begin{displaymath}
\ba
&[0,0,0], \quad [0,e_j,0] \;\;\forall j\in\F,\quad [0,e_j,e_{ij}] \;\; \forall i\in\cS\;\forall j\in\F, \\
&[\epsilon e_{ij},e_j,e_{ij}] \;\; \forall i\in\cS\;\forall j\in\F,  
\ea
\end{displaymath}
where $\epsilon$ is a sufficient small positive constant. 
In particular, $\epsilon<\min\{d_{ij},c_j\}$.
It is easy to verify the above $(2m+1)n+1$ points are affinely 
independent points in $\cchiL$, and hence $\dim\cchiL=(2m+1)n$.
Using a similar argument, we can show that $\dim\CIJ=(2|I|+1)|J|$. 
We now focus on proving the theorem.

 We refer the inequality $\sum_{j\in J}\alpha_jy_j+\sum_{i\in I}\sum_{j\in J}(x_{ij}+\beta_{ij}q_{ij})\le \gamma$
as $ineq$ in the proof. Since $ineq$ defines a facet of $\CIJ$, 
there exist $(2|I|+1)|J|$ affinely independent points in $\CIJ$ 
that are on the plane defined by $ineq$.
We denote this set of the $(2|I|+1)|J|$ affinely independent points as $V$. 
We will construct additional $(2m+1)n-(2|I|+1)|J|$ linearly 
independent points of $\cchiL$ that are on the plane
defined by $ineq$. Let $[x,y,q]$ be an arbitrary point in $V$, 
we first construct the following set of points:
\begin{displaymath}
\begin{aligned}
&V^1= \Set*{[x,y,q]+[0,e_j,0]}{ \forall j\in[n]\setminus J }  \\
&\qquad \cup \Set*{[x,y,q]+[0,e_j,e_{ij}]}{\forall i\in[m]\setminus I,\;\forall j\in[n]\setminus J} \\
&\qquad \cup \Set*{[x,y,q]+[\epsilon e_{ij},e_j,e_{ij}]}{\forall i\in[m]\setminus I, \;\forall j\in[n]\setminus J}, 
\end{aligned} 
\end{displaymath}
with $\epsilon>0$ sufficient small. 
We can verify that all points in $V^1$ are linearly independent points of $\cchiL$ 
that are on the plane defined by $ineq$, and $|V^1|=(2m-2|I|+1)(n-|J|)$. 
Let $[x^{(j)},y^{(j)},q^{(j)}]\in V$ for every $j\in J$ be the vector satisfying the condition (2).
We construct the following set of points: 
\begin{displaymath}
\begin{aligned}
V^2=&\Set*{[x^{(j)},y^{(j)},q^{(j)}]+[0,0,e_{ij}]}{\forall i\in [m]\setminus I,\;\forall j\in J  }\\
&\cup\Set*{[x^{(j)},y^{(j)},q^{(j)}]
 +[\epsilon e_{ij},0,e_{ij}]}{\forall i\in [m]\setminus I,\;\forall j\in J },
\end{aligned}
\end{displaymath}
with sufficient small $\epsilon>0$. 
One can verify that all points in $V^2$ are linearly independent points of $\CIJ$ 
that are on the plane defined by $ineq$, and $V^2$ contains $2(m-|I|)|J|$ linearly independent points.
Finally, we construct the following set of points:
\begin{displaymath}
\begin{aligned}
V^3=&\Set*{[\tilde{x}^{(i)},\tilde{y}^{(i)},\tilde{q}^{(i)}]+[0,e_j,\epsilon e_{ij}]}{\forall i\in I,\; \forall j\in[n]\setminus J} \\
&\cup\Set*{[\tilde{x}^{(i)},\tilde{y}^{(i)},\tilde{q}^{(i)}]+[\epsilon^{\prime}e_{ij},e_j,\epsilon e_{ij}]}{\forall i\in I,\;\forall j\in[n]\setminus J},
\end{aligned}
\end{displaymath}
where $[\tilde{x}^{(i)},\tilde{y}^{(i)},\tilde{q}^{(i)}]$ is the point satisfying the condition (3)
for each $i\in I$, and $\epsilon, \epsilon^{\prime}>0$ are sufficient small numbers. 
Again, it is straightforward to verify that all points in $V^3$ are linearly independent points of $\cchiL$
that are on the plane defined by $ineq$, and $|V^3|=2|I|(n-|J|)$. 
By looking at the perturbation terms involved in $V^1$, $V^2$ and $V^3$
we see that the vectors in $V^1$, $V^2$ and $V^3$ are linearly independent.
Furthermore, since $\textrm{span}\{V^1,V^2,V^3\}\cap\textrm{span}(V)=\{[0,0,0]\}$,
then by Lemma~\ref{lem:aff-indpt}, all vectors in $V\cup V^1\cup V^2\cup V^3$ are affinely independent.
The number of points we have constructed is
\begin{displaymath}
\begin{aligned}
|V|+|V^1|+|V^2|+|V^3|&=(2|I|+1)|J|+(2m-2|I|+1)(n-|J|) \\
&\quad+2(m-|I|)|J|+2|I|(n-|J|) \\
&=2mn+n,
\end{aligned}
\end{displaymath}
which shows that the $ineq$ defines a facet of $\cchiL$.
\end{proof}

\begin{theorem}\label{thm:valid-ineq-Z}
Suppose $d_{ij}>0$ for all $i\in[m],\in[n]$. The following hold:
\newline
\emph{(a)} For any subsets of indices $I\subseteq[m]$, 
$J\subset[n]$ and index $j^{\prime}\in[n]\setminus J$  
the following inequality is valid for $\cchiL$
\begin{equation}\label{eqn:Z-valid-ineq-1}
\sum_{i\in I}\sum_{j\in J\cup\{j^{\prime}\}}x_{ij} -\Big(\sum_{i\in I}\underset{j\in[n]}{\emph{max}}\;d_{ij} -\sum_{j\in J} c_j\Big)y_{j^{\prime}} 
\le \sum_{j\in J}c_j.
\end{equation}
\newline
\emph{(b)} Let $I\subseteq[m]$ and $J\subset[n]$ be two subset of indices, 
if $\sum_{j\in J}c_j\le \sum_{i\in I}\emph{max}_{j\in J}d_{ij}$, 
then the following inequality is valid for $\cchiL$
\begin{equation}\label{eqn:Z-valid-ineq-2}
\sum_{i\in I}\sum_{j\in[n]}x_{ij}
-\sum_{j^{\prime}\in[n]\setminus J}\Big(\sum_{i\in I}\underset{j\in J\cup\{j^{\prime}\}}{\emph{max}}\;d_{ij} -\sum_{j\in J} c_j\Big)y_{j^{\prime}} 
\le \sum_{j\in J}c_j.
\end{equation}
\end{theorem}
\begin{proof}
Part (a): It suffices to show the validness for the two cases: $y_{j^{\prime}}=0$ and $y_{j^{\prime}}=1$.
If $y_{j^{\prime}}=0$, then we have
\begin{equation}\label{eqn:Z-valid-proof1}
\begin{aligned}
&\sum_{i\in I}\sum_{j\in J\cup\{j^{\prime}\}}x_{ij} -\Big(\sum_{i\in I}\underset{j\in[n]}{\textrm{max}}\;d_{ij} -\sum_{j\in J} c_j\Big)y_{j^{\prime}} \\
&=\sum_{i\in I}\sum_{j\in J}x_{ij}  
\le \sum_{j\in J}\sum_{i\in [m]}x_{ij} \le \sum_{j\in J}c_j, 
\end{aligned}
\end{equation}
where we use $y_{j^{\prime}}=0$, $y_j\le 1$ and the constraint 
$\sum_{i\in[m]}x_{ij}\le c_jy_j$ to get the above inequalities.
For the case of $y_{j^{\prime}}=1$, we have 
\begin{equation}\label{eqn:Z-valid-proof2}
\begin{aligned}
&\sum_{i\in I}\sum_{j\in J\cup\{j^{\prime}\}}x_{ij} -\Big(\sum_{i\in I}\underset{j\in[n]}{\textrm{max}}\;d_{ij} -\sum_{j\in J} c_j\Big)y_{j^{\prime}} \\
&=\sum_{i\in I}\sum_{j\in J\cup\{j^{\prime}\}}x_{ij} -\sum_{i\in I}\underset{j\in[n]}{\textrm{max}}\;d_{ij} +\sum_{j\in J} c_j \\
&\le \sum_{i\in I}\sum_{j\in [n]}x_{ij}-\sum_{i\in I}\underset{j\in [n]}{\textrm{max}}\;d_{ij} +\sum_{j\in J} c_j \\
&\le \sum_{i\in I}\sum_{j\in [n]}d_{ij}q_{ij} -\sum_{i\in I}\underset{j\in [n]}{\textrm{max}}\;d_{ij} +\sum_{j\in J} c_j \le \sum_{j\in J} c_j,
\end{aligned}
\end{equation}
where we use $y_{j^{\prime}}=1$, and the constraint 
$\sum_{j\in[n]}x_{ij}\le\sum_{j\in[n]}d_{ij}q_{ij}$ in
\eqref{eqn:Z-valid-proof2}. 

Part (b): For a feasible point $[x,y,q]$, let $J^{\prime}=\Set*{j\in[n]\setminus J}{y_j=1}$.
For any $i\in I$, let $\pi(i)=\textrm{argmax}_{j\in J\cup J^\prime}d_{ij}$ where ties are broken arbitrarily. 
Let $I_{j}=\Set*{i\in I}{\pi(i)=j}$ for every $j\in J\cup J^\prime$, and by definition
the subsets $\{I_j\}_{j\in J\cup J^\prime}$ form a partition of $I$. 
If $|J^{\prime}|\le 1$, then \eqref{eqn:Z-valid-ineq-2} reduces to \eqref{eqn:Z-valid-ineq-1}.
Therefore without loss of generality, we assume $|J^{\prime}|\ge 2$. 
The expression on the left can be relaxed as
\begin{equation}\label{eqn:valid-ineq-LHS}
\begin{aligned}
&\sum_{i\in I}\sum_{j\in[n]}x_{ij}
-\sum_{j^{\prime}\in[n]\setminus J}\Big(\sum_{i\in I}\underset{j\in J\cup\{j^{\prime}\}}{\textrm{max}}\;d_{ij} -\sum_{j\in J} c_j\Big)
y_{j^{\prime}} \\
&\le\sum_{i\in I}\sum_{J\cup J^\prime}d_{ij}q_{ij}-\sum_{j^{\prime}\in J^{\prime}}\Big(\sum_{i\in I}\underset{j\in J\cup\{j^{\prime}\}}{\textrm{max}}\;d_{ij} -\sum_{j\in J} c_j\Big).
\end{aligned}
\end{equation}
Notice that we have the following inequalities hold
\begin{equation}\label{eqn:valid-ineq-R1}
\begin{aligned}
&\sum_{i\in I}\sum_{j\in J\cup J^\prime}d_{ij}q_{ij} \le \sum_{i\in I}d_{i\pi(i)}
=\sum_{j\in J\cup J^\prime} \sum_{i\in I_j} d_{ij} \\
&\le \sum_{j^*\in J}\sum_{i\in I_{j^*}}\underset{j\in J}{\textrm{max}}\; d_{ij}
  + \sum_{j^*\in J^{\prime}}\sum_{i\in I_{j^*}}  d_{ij^*}. 
\end{aligned}
\end{equation}
Furthermore, the second term in \eqref{eqn:Z-valid-ineq-2} is evaluated as 
\begin{equation}\label{eqn:second-term-in-ineq}
\begin{aligned}
&\sum_{j^{\prime}\in[n]\setminus J}\Big(\sum_{i\in I}\underset{j\in J\cup\{j^{\prime}\}}{\textrm{max}}\;d_{ij} -\sum_{j\in J} c_j\Big)y_{j^{\prime}} =\sum_{j^{\prime}\in J^{\prime}}\Big(\sum_{i\in I}\underset{j\in J\cup\{j^{\prime}\}}{\textrm{max}}\; d_{ij} -\sum_{j\in J} c_j\Big) \\
&=\sum_{j^{\prime}\in J^{\prime}}\sum_{j^*\in J\cup J^\prime}\sum_{i\in I_{j^*}} \underset{j\in J\cup\{j^{\prime}\}}{\textrm{max}}\; d_{ij}
- |J^{\prime}|\sum_{j\in J}c_j. 
\end{aligned}
\end{equation}
The first term in \eqref{eqn:second-term-in-ineq} can be decomposed as
\begin{equation}\label{eqn:valid-ineq-R2}
\begin{aligned}
&\sum_{j^{\prime}\in J^{\prime}}\sum_{j^*\in J\cup J^\prime}\sum_{i\in I_{j^*}} \underset{j\in J\cup\{j^{\prime}\}}{\textrm{max}}\; d_{ij}\\
&=\sum_{j^{\prime}\in J^{\prime}}\sum_{j^*\in J}\sum_{i\in I_{j^*}}\underset{j\in J\cup\{j^{\prime}\}}{\textrm{max}}\; d_{ij}
 + \sum_{j^{\prime}\in J^{\prime}}\sum_{j^*\in J^{\prime}}\sum_{i\in I_{j^*}} \underset{j\in J\cup\{j^{\prime}\}}{\textrm{max}}\; d_{ij}
  \\
 &:=T_1+T_2,
\end{aligned}
\end{equation}
where $T_1$ and $T_2$ represent the first and second term 
on the right side of \eqref{eqn:valid-ineq-R1}, respectively.
The term $T_1$ is evaluated as
\beq\label{eqn:T1}
\ba
&T_1=\sum_{j^{\prime}\in J^{\prime}}\sum_{j^*\in J}\sum_{i\in I_{j^*}}\underset{j\in J\cup\{j^{\prime}\}}{\textrm{max}}\; d_{ij}&=\sum_{j^{\prime}\in J^{\prime}}\sum_{j^*\in J}\sum_{i\in I_{j^*}}\underset{j\in J}{\textrm{max}}\; d_{ij} \\
&=|J^{\prime}|\sum_{j^*\in J}\sum_{i\in I_{j^*}}\underset{j\in J}{\textrm{max}}\; d_{ij},
\ea
\eeq
where we use the fact that for any $j^*\in J$, $i\in I_{j^*}$ and $j^{\prime}\in[n]\setminus J$, 
we have $d_{ij^*}=\max_{j\in J\cup\{j^{\prime}\}}d_{ij}=\max_{j\in J}d_{ij}$.
The term $T_2$ can be lower bounded as:
\begin{equation}\label{eqn:T2}
\begin{aligned}
 &T_2=\sum_{j^*\in J^{\prime}}\sum_{i\in I_{j^*}}\sum_{j^{\prime}\in J^{\prime}} \underset{j\in J\cup\{j^{\prime}\}}{\textrm{max}}\; d_{ij} \\
 &=\sum_{j^*\in J^{\prime}}\sum_{i\in I_{j^*}} \max_{j\in J\cup\{j^*\}} d_{ij} 
 + \sum_{j^*\in J^{\prime}}\sum_{i\in I_{j^*}} \sum_{j^{\prime}\in J^{\prime}\setminus\{j^*\}}\underset{j\in J\cup\{j^{\prime}\}}{\textrm{max}}\; d_{ij} \\
 &=\sum_{j^*\in J^{\prime}}\sum_{i\in I_{j^*}} d_{ij^*} 
 + \sum_{j^*\in J^{\prime}}\sum_{i\in I_{j^*}}\sum_{j^{\prime}\in J^{\prime}\setminus\{j^*\}} \underset{j\in J\cup\{j^{\prime}\}}{\textrm{max}}\; d_{ij} \\
 &\ge \sum_{j^*\in J^{\prime}}\sum_{i\in I_{j^*}} d_{ij^*}
 +  (|J^{\prime}|-1) \sum_{j^*\in J^{\prime}}\sum_{i\in I_{j^*}} \underset{j\in J}{\textrm{max}}\; d_{ij}.
\end{aligned}
\end{equation}
Substituting \eqref{eqn:T1} and \eqref{eqn:T2} into \eqref{eqn:valid-ineq-R2} yields the following inequalities:
\begin{equation}\label{eqn:R2-final}
\begin{aligned}
&\sum_{j^{\prime}\in J^{\prime}}\sum_{j^*\in J\cup J^\prime}\sum_{i\in I_{j^*}} \underset{j\in J\cup\{j^{\prime}\}}{\textrm{max}}\; d_{ij}
\ge |J^\prime|\sum_{j^*\in J}\sum_{i\in I_{j^*}}\underset{j\in J}{\textrm{max}}\; d_{ij} 
+ \sum_{j^*\in J^{\prime}}\sum_{i\in I_{j^*}} d_{ij^*} \\
&\qquad + (|J^{\prime}|-1) \sum_{j^*\in J^{\prime}}\sum_{i\in I_{j^*}} \underset{j\in J}{\textrm{max}}\; d_{ij} \\
&=\sum_{j^*\in J}\sum_{i\in I_{j^*}}\underset{j\in J}{\textrm{max}}\; d_{ij}+\sum_{j^*\in J^{\prime}}\sum_{i\in I_{j^*}} d_{ij^*}
+(|J^{\prime}|-1) \sum_{j^*\in J\cup J^{\prime}}\sum_{i\in I_{j^*}} \underset{j\in J}{\textrm{max}}\; d_{ij} \\
&=\sum_{j^*\in J}\sum_{i\in I_{j^*}}\underset{j\in J}{\textrm{max}}\; d_{ij}+\sum_{j^*\in J^{\prime}}\sum_{i\in I_{j^*}} d_{ij^*}
+(|J^{\prime}|-1) \sum_{i\in I} \underset{j\in J}{\textrm{max}}\; d_{ij},
\end{aligned}
\end{equation}
where we use the fact that $\Set*{I_{j^*}}{j^*\in J\cup J^\prime}$ form a partition of $I$. 
Substituting \eqref{eqn:R2-final} into \eqref{eqn:second-term-in-ineq} yields
\beq\label{eqn:valid-ineq-R1-final}
\ba
&\sum_{j^{\prime}\in[n]\setminus J}\Big(\sum_{i\in I}\underset{j\in J\cup\{j^{\prime}\}}{\textrm{max}}\;d_{ij} -\sum_{j\in J} c_j\Big)y_{j^{\prime}}
\ge \sum_{j^*\in J}\sum_{i\in I_{j^*}}\underset{j\in J}{\textrm{max}}\; d_{ij}+\sum_{j^*\in J^{\prime}}\sum_{i\in I_{j^*}} d_{ij^*} \\
&\qquad\qquad +(|J^{\prime}|-1) \sum_{i\in I} \underset{j\in J}{\textrm{max}}\; d_{ij}
-|J^\prime|\sum_{j\in J}c_j \\
&\ge \sum_{j^*\in J}\sum_{i\in I_{j^*}}\underset{j\in J}{\textrm{max}}\; d_{ij}
+\sum_{j^*\in J^{\prime}}\sum_{i\in I_{j^*}} d_{ij^*}-\sum_{j\in J}c_j
\ea
\eeq
where we use the assumption $\sum_{i\in I}\textrm{max}_{j\in J}d_{ij}-\sum_{j\in J}c_j\ge 0$.
Substituting \eqref{eqn:valid-ineq-R1} and \eqref{eqn:valid-ineq-R1-final}  into \eqref{eqn:valid-ineq-LHS} yields
\begin{equation}
\sum_{i\in I}\sum_{j\in[n]}x_{ij}
-\sum_{j^{\prime}\in[n]\setminus J}\left(\sum_{i\in I}\underset{j\in J\cup\{j^{\prime}\}}{\textrm{max}}\; d_{ij} -\sum_{j\in J} c_j\right)y_{j^{\prime}}
\le \sum_{j\in J} c_j.
\end{equation}
This shows that \eqref{eqn:Z-valid-ineq-2} is valid for $\cchiL$.
\end{proof}

\begin{theorem}\label{thm:facet-critical-IJ}
Let $I\subseteq[m]$ and $J\subseteq[n]$.
Suppose the following conditions are satisfied
\begin{equation}\label{eqn:facet-cond}
\ba
&\sum_{j\in J}c_j>\sum_{i\in I}\max_{j\in J}d_{ij},\qquad \max_{j\in[n]}d_{ij}=\max_{j\in J}d_{ij}\quad\forall i\in I, \\
&\sum_{j\in J\setminus\{j_0\}}c_j<\sum_{i\in I}\max_{j\in J\setminus\{j_0\}}d_{ij} \qquad\forall j_0\in J.
\ea
\end{equation}
Then the following inequality defines a facet of $\cchiL$
\begin{equation}\label{eqn:facet-def-ineq}
\sum_{i\in I}\sum_{j\in J}x_{ij} 
+ \sum_{j\in J}\Big(\sum_{j^{\prime}\in J\setminus\{j\}}c_{j^{\prime}} 
-\sum_{i\in I}\max_{j^{\prime}\in J}d_{ij^\prime}\Big) y_j
\le (|J|-1)\Big(\sum_{j\in J}c_j-\sum_{i\in I}\max_{j^{\prime}\in J}d_{ij^\prime}\Big).
\end{equation}
\end{theorem}
\begin{proof}
We will first provide an intuition of how the coefficients of the inequality \eqref{eqn:facet-def-ineq}
are determined given the condition in \eqref{eqn:facet-cond}. Since we restrict the decision
variables induced by the subsets $I$ and $J$, a facet-defining valid inequality is likely in the form of
\beq
\sum_{i\in I}\sum_{j\in J}x_{ij}+\sum_{j\in J}a_jy_j \le b.
\eeq 
We need to determine the coefficients $a_j$ and $b$. Notice that
there are $|J|+1$ coefficients to be determined, and we will set up
$|J|+1$ linear equations corresponding to the following $|J|+1$ situations:
(1) opening service centers at all locations in the subset $J$,
(2) opening service centers at all locations in $J\setminus\{j\}$ for all $j\in J$. 
Notice that (2) contains $|J|$ different situations referred as situation $(2.j)$.
In the first situation, all locations in $J$ have service centers. So the maximum
value of $\sum_{i\in I}\sum_{j\in J}x_{ij}$ is $\sum_{i\in I}\max_{j\in J}d_{ij}$, i.e., 
the aggregated demand constraint is binding.
Similarly, in the situation $(2.j_0)$,  locations in $J\setminus\{j_0\}$ have service centers,
and the maximum value of $\sum_{i\in I}\sum_{j\in J}x_{ij}$ is 
$\sum_{j\in J\setminus\{j_0\}}c_j$,
i.e., the aggregated capacity constraint is binding. These observations
help us set up the following linear equations to determine $a_j$ and $b$:
\begin{subequations}\label{eqn:lin-sys-coef}
\begin{align}
&\sum_{i\in I}\max_{j\in J}d_{ij}+\sum_{j\in J}a_j = b, \label{eqn:D+a=b} \\
&\sum_{j^{\prime}\in J\setminus\{j\}}c_{j^\prime}
+ \sum_{j^{\prime}\in J\setminus\{j\}}a_{j^\prime} = b \qquad\forall j\in J. \label{eqn:C+a=b}
\end{align}
\end{subequations}
By solving the above linear equation system, we can determine the values
of all coefficients:
\bdm
\ba
&b=(|J|-1)\Big(\sum_{j\in J}c_j-\sum_{i\in I}\max_{j^{\prime}\in J}d_{ij^\prime}\Big), \\
&a_j=\sum_{j^{\prime}\in J\setminus\{j\}}c_{j^{\prime}} 
-\sum_{i\in I}\max_{j^{\prime}\in J}d_{ij^\prime} \qquad\forall j\in J,
\ea
\edm
which leads to the inequality \eqref{eqn:facet-def-ineq}.

Now we give a formal proof of that \eqref{eqn:facet-def-ineq} 
is a facet-defining valid inequality of $\cchiL$.
To see the validness, we divide the analysis into four cases: 
(a) $y_j=1$ for all $j\in J$;
(b) $y_j=1$ and $y_{j^{\prime}}=0$ for all $j^{\prime}\in J\setminus\{j\}$; 
(c) there exists a $J_0\subset J$ satisfying $2\le |J_0|\le |J|-1$ and
$y_j=1$ for all $j\in J\setminus J_0$ and $y_j=0$ for all $j\in J_0$;
(d) $y_j=0$ for all $j\in J$. The validness of \eqref{eqn:facet-def-ineq} clearly
holds in the cases (a) and (b) by the way of determining the coefficients.
For the case (d), \eqref{eqn:facet-def-ineq} becomes
\bdm
\sum_{i\in I}\sum_{j\in J}x_{ij} 
\le (|J|-1)\Big(\sum_{j\in J}c_j-\sum_{i\in I}\max_{j^{\prime}\in J}d_{ij^\prime}\Big).
\edm
It is valid because in this case we have $\sum_{i\in I}\sum_{j\in J}x_{ij}=0$ 
and \newline
$c_j-\sum_{i\in I}\max_{j^{\prime}\in J}d_{ij^\prime}>0$ by assumption.
It remains to verify the validness in the case (c). Substitute $y_j=1$ for $j\in J\setminus J_0$
into \eqref{eqn:facet-def-ineq}, and notice that we have
\bdm
\ba
&(|J|-1)\Big(\sum_{j\in J}c_j-\sum_{i\in I}\max_{j^{\prime}\in J}d_{ij^\prime}\Big)
-\sum_{j\in J\setminus J_0}\Big(\sum_{j^{\prime}\in J\setminus\{j\}}c_{j^{\prime}} 
-\sum_{i\in I}\max_{j^{\prime}\in J}d_{ij^\prime}\Big) \\
&=\sum_{j_0\in J_0}\sum_{j\in J\setminus\{j_0\}}c_j-(|J_0|-1)\sum_{i\in I}\max_{j^{\prime}\in J}d_{ij^\prime}\\
&\ge\sum_{j_0\in J_0}\sum_{j\in J\setminus\{j_0\}}c_j-(|J_0|-1)\sum_{j\in J}c_j \\
&=(|J_0|-1)\sum_{j\in J_0}c_j + |J_0|\sum_{j\in J\setminus J_0}c_j - (|J_0|-1)\sum_{j\in J}c_j \\
&= \sum_{j\in J\setminus J_0}c_j \ge \sum_{i\in I}\sum_{j\in J}x_{ij},
\ea
\edm
where the assumed condition is used to get the first inequality above.
It proves the validness in this case.

To show that \eqref{eqn:facet-def-ineq} defines a facet of $\cchiL$,
we will first show that it defines a facet of $\CIJ$ (see \eqref{def:ZIJ}),
and apply Theorem~\ref{thm:facet-extension} to conclude that it also
defines a facet of $\cchiL$ via verifying the three conditions 
in Theorem~\ref{thm:facet-extension} are satisfied. Let 
$V=\Set*{[x,y,q]\in\CIJ}{[x,y,q]\textrm{ satisfies \eqref{eqn:facet-def-ineq} as equality}}$.
Clearly, there exists a facet of $\CIJ$ containing all points in $V$,
and suppose this facet is represented in the following general form:
\beq\label{eqn:gen-facet-ZIJ}
\sum_{i\in I}\sum_{j\in J}\alpha_{ij}x_{ij} + \sum_{j\in J}\beta_jy_j
+\sum_{i\in I}\sum_{j\in J}\gamma_{ij}q_{ij}\le \lambda.
\eeq
It suffices to show that the above coefficients are different from the
coefficients in \eqref{eqn:facet-def-ineq} by a positive constant.
Note that all points in $V$ should satisfy \eqref{eqn:gen-facet-ZIJ}
as an equality. We are going to propose some specific points from $V$
and plug them into \eqref{eqn:gen-facet-ZIJ} to determine the coefficients.
First, we notice that there must exist a vector $[\sum_{j\in J}x^0,e_j,q^0]\in V$
such that $x^0_{ij}>0$, $0<q^0_{ij}<1$ for all $(i,j)\in I\times J$, 
and $\sum_{i\in I}x^0_{ij}<c_j$ for all
$j\in J$. It can be deduced that the points
$[x^0+\epsilon e_{ij_1}-\epsilon e_{ij_2}, \sum_{j\in J}e_j, q^0]$ 
for all $i\in I,\;j_1,j_2\in J,\; j_1\neq j_2$, and 
$[\sum_{j\in J}e_j,x^0+\epsilon e_{i_1j}-\epsilon e_{i_2j},q^0]$
for all $i_1,i_2\in I,\; j\in J$ are also in $V$ for some sufficient
small $\epsilon$. Substituting these points into \eqref{eqn:gen-facet-ZIJ}
as an equality, we conclude that $\alpha_{ij_1}=\alpha_{ij_2}$
and $\alpha_{i_1j}=\alpha_{i_2j}$, which implies that 
$\alpha_{ij}=\alpha\;\forall (i,j)\in I\times J$ for some $\alpha>0$.
Without loss of generality, we can set $\alpha=1$ since the inequality
is invariant under multiplication of a positive constant.
Similarly, the exists a point $[x^1,\sum_{j\in J\setminus\{j_0\}}e_j,q^1]\in V$
such that $0<q^1_{ij}<1$ for all $(i,j)\in I\times(J\setminus\{j_0\})$.
It can be deduced that the points 
$[x^1,\sum_{j\in J\setminus\{j_0\}}e_j,q^1+\epsilon e_{i_1j_1}]\;\forall(i_1,j_1)\in I\times(J\setminus\{j_0\})$
are all in $V$.
Substituting this point into \eqref{eqn:gen-facet-ZIJ} as an equality,
we conclude that $\gamma_{ij}=0$ for all $(i,j)\in I\times(J\setminus\{j_0\})$.
Since $j_0$ is arbitrary, we conclude that $\gamma_{ij}=0$ for all $(i,j)\in I\times J$.
Therefore, \eqref{eqn:gen-facet-ZIJ} becomes
\beq\label{eqn:gen-facet-ZIJ-simp}
\sum_{i\in I}\sum_{j\in J}x_{ij} + \sum_{j\in J}\beta_jy_j\le \lambda.
\eeq
Clearly, since the previously considered point $[x^0,\sum_{j\in J}e_j,q^0]\in V$
should satisfy $\sum_{i\in I}\sum_{j\in J}x^0_j=\sum_{i\in I}\max_{j\in J}d_{ij}$
(this can be obtained by substituting the point into \eqref{eqn:facet-def-ineq} as an equality).
Then substituting this point into \eqref{eqn:gen-facet-ZIJ-simp} as an equality
gives an equation in the same form of \eqref{eqn:D+a=b}.
Similarly, substituting the point $[x^1,\sum_{j\in J\setminus\{j_0\}}e_j,q^1]\in V$
into \eqref{eqn:gen-facet-ZIJ-simp} as an equality gives an equation in the same
form of \eqref{eqn:C+a=b}. Therefore the to be determined coefficients $[\beta,\gamma]$
are the solution of \eqref{eqn:lin-sys-coef}, which shows that the facet defining valid inequality
\eqref{eqn:gen-facet-ZIJ-simp} of $\CIJ$ is exactly \eqref{eqn:facet-def-ineq}.

Now we show that \eqref{eqn:facet-def-ineq} is also a facet defining inequality 
of $\cchiL$ using Theorem~\ref{thm:facet-extension}. 
To achieve this, we first need to show that \eqref{eqn:facet-def-ineq}
is valid for $\cchiL$. The validness of \eqref{eqn:facet-def-ineq} for $\cchiL$
is based on the following two observations: if a subset 
$J^{\prime}\subseteq[n]\setminus J$ of service centers are opened,
the maximum attracted demand from a customer site $i$ 
is at most $\max_{j\in J\cup J^\prime}d_{ij}\le\max_{j\in[n]}d_{ij}$.
But $\max_{j\in[n]}d_{ij}=\max_{j\in J}d_{ij}$ by assumption, which
implies that once all places in $J$ have service centers, the aggregated demand
constraint $\sum_{i\in I}\sum_{j\in J}x_{ij}\le\sum_{i\in I}\max_{j\in J}d_{ij}$
is valid and it is independent of whether or not other places having service centers. 
Similarly, in the case that places at $J\setminus\{j_0\}$ have service centers,
the aggregated capacity constraint dominates the aggregated demand constraint
which is also independent of service center locations at other places due to the fact
$\max_{j\in J\setminus\{j_0\}}d_{ij}\le\max_{j\in J\cup J^{\prime}\setminus\{j_0\}}d_{ij}$
for any $J^{\prime}\subseteq[n]\setminus J$. It remains to 
show that there exist $(2|I|+1)|J|$ affinely independent vectors in $\CIJ$ 
that satisfy the three conditions in Theorem~\ref{thm:facet-extension}.
Since we have already proved that \eqref{eqn:facet-def-ineq}
is a facet defining inequality of $\CIJ$, there must exist a set $V$ of
$(2|I|+1)|J|$ affinely independent vectors in $\CIJ$ that 
satisfy \eqref{eqn:facet-def-ineq} as an equality.
Furthermore, we can make sure that $V$ includes the following two vectors:
$[x^0,\sum_{j\in J}e_j, q^0]$ and $[x^1,\sum_{j\in J\setminus\{j_0\}}e_j, q^1]$
where the two vectors are introduced in the previous analysis of this proof.
Notice that $[x^0,\sum_{j\in J}e_j, q^0]$ satisfies the condition (2) 
of Theorem~\ref{thm:facet-extension} for all $j\in J$, and
$[ x^1,\sum_{j\in J\setminus\{j_0\}}e_j, q^1]$ satisfies the condition (3) 
of Theorem~\ref{thm:facet-extension} for each $i\in I$.
This concludes the proof.
\end{proof}

Theorem~\ref{thm:valid-ineq-Z} and \ref{thm:facet-critical-IJ} provide
valid and facet-defining inequalities for $\conv(\chi^L)$. Notice that these 
inequalities do not involve variables $q$ introduced to lift $\chi$. 
This observation raises a question: are these inequalities also valid and facet-defining for $\conv(\chi)$?
The following proposition gives an affirmative answer to this question.
\begin{proposition}
Let $[x,y]\in\mbR^{m+n}$, $P$ be a full dimensional polytope in $\mbR^m$ represented
by a linear-inequality system of variables $x$, and $Q$ be a full dimensional polytope in $\mbR^{m+n}$
represented by a linear-inequality system of variables $x,y$. Suppose $P=\emph{proj}_x Q$.
If $\alpha^\top x+\beta\le 0$ is a valid (facet-defining) inequality of $Q$, it is also a valid (facet-defining)
inequality of $P$.
\end{proposition}
\begin{proof}
(1) Suppose $\alpha^\top x+\beta\le 0$ is a valid inequality of $Q$. 
Due to the projection relationship, for any $x^\prime\in P$,
there exists a $[x^\prime,y^\prime]\in Q$. Since $\alpha^\top x+\beta\le 0$
is valid for $Q$, we have $\alpha^\top x^\prime+\beta\le 0$, which implies
that $\alpha^\top x+\beta\le 0$ is valid for $P$. (2) Suppose $\alpha^\top x+\beta\le 0$
is a facet-defining inequality of $Q$. We focus on the case that the facet does not
contain the point $[0,0]$. By definition, there exists $m+n$ linearly independent
points $\{[x^k,y^k]:\;k\in[m+n]\}$ satisfying $\alpha^\top x^k+\beta=0$. Since every $y^k$
is a $n$-dimensional vector, the set $\{x^k:\;k\in[m+n]\}$ contains $m$ linearly independent
vectors. In other words, there exist $m$ linearly independent vectors on the hyperplane 
$\alpha^\top x+\beta=0$, which shows that $\alpha^\top x+\beta\le 0$ defines a facet of $P$.
\end{proof}
\begin{corollary}
The inequality \eqref{eqn:Z-valid-ineq-1} is valid for $\conv(\chi)$, 
the inequality \eqref{eqn:Z-valid-ineq-2} is valid for $\conv(\chi)$ given
that the condition in Theorem~\ref{thm:valid-ineq-Z} is satisfied,
and the inequality \eqref{eqn:facet-def-ineq} defines a facet of $\conv(\chi)$
given that the conditions in Theorem~\ref{thm:facet-critical-IJ} is satisfied.
\end{corollary}

\section{Conclusion}
The facet-defining inequalities given in Theorem~\ref{thm:valid-ineq-Z} hold 
when the parameters of demand and capacity satisfy a certain condition for a subset
of client locations and candidate facility locations. The same spirit can be used to
construct various conditions that lead to different families of facet-defining inequalities. 
Identifying the subsets that satisfy the conditions is also a NP-hard problem, 
which prevents a systematic application of these inequalities in computation at low cost
of computing power. This indicates that the facet-defining inequalities identified in this work is of 
more theoretical interest for understanding the facility-location polyhedral structure
under competing constraints of capacity and demand.

\section*{Acknowledgement}
\noindent This research was partially supported by the ONR grant N00014-18-1-2097.

\bibliography{reference}
\end{document}